\documentclass[a4paper,11pt]{amsart}

\usepackage[english]{babel}
\usepackage[utf8]{inputenc}
\usepackage[T1]{fontenc}
\usepackage{comment}
\usepackage{amssymb}
\usepackage{mathrsfs}
\usepackage{hyperref}
\usepackage[usenames,dvipsnames]{xcolor}

\newcommand{\R}{\mathbb{R}}

\newcommand{\Z}{\mathbb{Z}}

\newcommand{\Hek}{\mathbb{H}^{k}}
\newcommand{\Rk}{\mathbb{R}^{2k}}
\newcommand{\Sk}{\mathbb{S}^{2k-1}}

\newcommand{\calHk}{\mathcal{H}^{2k+1}}
\newcommand{\Se}{\mathbb{S}}

\newcommand{\V}{\mathbb{V}}
\newcommand{\LL}{\mathbb{L}}

\newcommand{\calH}{\mathcal{H}}
\newcommand{\calV}{\mathcal{V}}
\newcommand{\calP}{\mathcal{P}}

\newcommand{\card}{\operatorname{card}}
\newcommand{\cone}{\operatorname{Cone}}
\newcommand{\diam}{\operatorname{diam}}
\newcommand{\dist}{\operatorname{dist}}

\newcommand{\BPiLG}{\operatorname{BPiLG}}
\newcommand{\BWGL}{\operatorname{BWGL}}

\numberwithin{equation}{section}

\theoremstyle{plain}
\newtheorem{theorem}[equation]{Theorem}
\newtheorem{proposition}[equation]{Proposition}
\newtheorem{lemma}[equation]{Lemma}

\theoremstyle{definition}
\newtheorem{definition}[equation]{Definition}

\theoremstyle{remark}
\newtheorem{remark}[equation]{Remark}

\begin{document}

\author{S\'everine Rigot}

\title{Quantitative notions of rectifiability in the Heisenberg groups}

\address{Universit\'e C\^ote d'Azur, CNRS, LJAD}

\email{severine.rigot@unice.fr}

\thanks{The author is partially supported by ANR grant ANR-15-CE40-0018.}


\begin{abstract} 
Several quantitative notions of rectifiability in the Heisenberg groups have emerged in the recent literature. In this paper we study the relationship between two of them, the big pieces of intrinsic Lipschitz graphs (BPiLG) condition and the bilateral weak geometric lemma (BWGL), and show that sets with BPiLG satisfy the BWGL. 
\end{abstract}

\maketitle


\section{Introduction} \label{sect:introduction}

Several notions of quantitative rectifiability in the Heisenberg groups have emerged in the recent literature. In this note we investigate the relationship between two of them, the big pieces of intrinsic Lipschitz graphs condition and the bilateral weak geometric lemma.

Quantitative notions of rectifiability in Euclidean spaces originate from works by G.~David and S.~Semmes, see~\cite{MR1251061} and the references therein. In this setting the validity of the bilateral weak geometric lemma is known to be one of the many characterizations of uniform rectifiability. It is also known that sets with big pieces of Lipschitz graphs are uniformly rectifiable, the former condition being actually strictly stronger than uniform rectifiability.

In the Heisenberg setting, intrinsic Lipschitz graphs, Definition~\ref{def:lip-graphs}, were introduced by B.~Franchi, R.~Serapioni and F.~Serra Cassano~\cite{MR2287539} and have been shown to be a relevant concept towards the development of geometric measure theory in the Heisenberg groups $\Hek$, and more generally in Carnot groups, see~\cite{MR3587666} and the references therein. The big pieces of intrinsic Lipschitz graphs (BPiLG) condition in $\Hek$, Definition~\ref{def:BPiLG}, first appeared in~\cite{CFO-arxiv}. Later on it was proved in~\cite{FOR-arxiv} that Semmes surfaces in $\Hek$, Definition~\ref{def:semmes-surface}, have BPiLG. The bilateral weak geometric lemma (BWGL) in $\Hek$, Definition~\ref{def:bwgl}, appeared implicitly in~\cite[Section~9]{MR3815462} and in a more formal form in~\cite{FOR-arxiv}. The motivation for this paper comes from the question about the relationship between BPiLG and BWGL in the Heisenberg setting. Our main result reads as follows~:

\begin{theorem} \label{thm:main}
Assume that $E\subset\Hek$ has big pieces of intrinsic Lipschitz graphs. Then $E$ satisfies the bilateral weak geometric lemma.
\end{theorem}

We also refer to Theorem~\ref{thm:BPILG-implies-BWGL} for a quantified version of Theorem~\ref{thm:main}. Our proof consists of two steps. First we verify that intrinsic Lipschitz graphs satisfy BWGL with suitable bounds, Theorem~\ref{thm:bwgl-for-intrinsic-Lip-graphs}. Second we prove a stability result for BWGL under the "big pieces functor", Theorem~\ref{thm:stability-bwgl}. We note that our proof of Theorem~\ref{thm:stability-bwgl} can be rephrased in the Euclidean setting, see Remark~\ref{rk:stability-Euclidean-BWGL}. To our knowledge the stability of BWGL in Euclidean spaces, Theorem~\ref{thm:stability-Euclidean-BWGL}, is also new and may have its own interest. 

The paper is organized as follows. In Section~\ref{sect:preliminaries} we sate our conventions about the Heisenberg groups. We recall the definition of intrinsic Lipschitz graphs in Section~\ref{sect:lip-graphs} and prove some of their properties for later use. The definition of the big pieces of intrinsic Lipschitz graphs condition can be found at the end of this section. In Section~\ref{sect:BWGL} we recall the definition of the bilateral weak geometric lemma and prove the two steps described above from which the proof of the more precise and quantified version of Theorem~\ref{thm:main}, see Section~\ref{sect:BPILG-implies-BWGL}, will follow.


\section{Preliminaries} \label{sect:preliminaries}

Throughout this paper we let $k \geq 1$ denote a fixed integer. We identify the Heisenberg group $\Hek$ with $\Rk \times \R$ equipped with the group law
\begin{displaymath}
(v,t) \cdot (v',t') =\left(v+v',t+t'+\omega(v,v')/2\right),\quad (v,t),(v',t')\in \Hek,
\end{displaymath}
where $\omega$ denotes the standard sympletic form on $\Rk$ given by $\omega(v,v') = \sum_{j=1}^k v_j v'_{j+k} - v_{j+k} v'_j$ for $v=(v_1,\dots,v_{2k})$, $v'=(v'_1,\dots,v'_{2k}) \in \Rk$. 

We denote by $(\delta_s)_{s>0}$ the family of dilations given by $\delta_s(v,t) = (sv,s^2 t)$. We recall that $(\delta_s)_{s>0}$ is a one-parameter group of group automorphisms.

We equip $\Hek$ with the Kor\'{a}nyi norm $\|\cdot\|$ and Kor\'{a}nyi distance $d$ defined by
 \begin{equation*}
 \|(v,t)\|= \sqrt[4]{|v|^4 + 16 t^2} \quad \text{and}\quad d(p,q)= \| q^{-1}\cdot p\| 
 \end{equation*}
where $|\cdot|$ denotes the Euclidean norm in $\Rk$.

We recall that the distance $d$ is homogeneous, meaning that $d$ is left-invariant, $d(p\cdot q, p\cdot q') = d(q,q')$ for all $p, q, q' \in \Hek$, and one-homogeneous with respect to the dilations $(\delta_s)_{s>0}$, $d(\delta_s(p), \delta_s(q)) = s d(p,q)$ for all $p,q\in\Hek$, $s>0$. 
 
Given $s \geq 0$ we denote by $\calH^s$ the $s$-dimensional Hausdorff measure in $(\Hek,d)$. We recall that $\calH^{2k+2}$ is a Haar measure on $\Hek$ and is $(2k+2)$-uniform. In particular $(\Hek,d)$ is a metric space of Hausdorff dimension $2k+2$. 

We denote by $B(p,r)=\{q \in \Hek : d(p,q) < r\}$ the open ball in $(\Hek,d)$ with center $p \in \Hek$ and radius $r > 0$. In this paper a ball will always refer to an open ball in $(\Hek,d)$.

We refer to~\cite{MR3587666} and the references therein for more details about the Heisenberg groups. 

We end this section with the well known definition of Ahlfors regular sets.

\begin{definition} Given $s\geq 0$ and $C\geq 1$ we say that $E\subset\Hek$ is $C$-Ahlfors regular with dimension $s$ if $E$ is closed and 
\begin{equation*}
C^{-1} r^s \leq \calH^{s} (E \cap B(p,r)) \leq C r^s \quad \text{for all } p\in E,\: r>0~.
\end{equation*}
\end{definition}

Since we will mostly be only concerned with Ahlfors regular sets with codimension one, we say in the rest of this paper that a subset of $\Hek$ is Ahlfors regular to mean that it is Ahlfors regular with dimension $2k+1$.


\section{Intrinsic Lipschitz graphs} \label{sect:lip-graphs}

The notion of intrinsic Lipschitz graphs in $\Hek$ originates from works by B.~Franchi, R.~Serapioni, and F.~Serra Cassano, see~\cite{MR3587666} and the references therein to which we refer for a thorough introduction. We stress that the class of intrinsic Lipschitz graphs given in Definition~\ref{def:lip-graphs} coincides with such a notion and that it also coincides with the one considered in~\cite[Section~2.3]{MR3815462}, see Remark~\ref{rk:FSSC-NY-lip-graphs}.

\subsection{Complementary subgroups} \label{subsect:complementary-subgroups}
We denote by $\Sk$ the Euclidean unit sphere in $\Rk$. Given $A\subset\Rk$, we denote by $A^\perp$ the linear subspace orthogonal to $A$ in the Euclidean sense in $\Rk$. Given $\nu \in \Sk$, we set
$$ \V_\nu= \R\nu^\perp \times \R \quad \text{and} \quad \LL_\nu=\R\nu \times \{0\}~.$$
These sets are complementary homogeneous subgroups of $\Hek$ which means that they are subgroups of $\Hek$, they are cones, that is, closed under the family of dilations $(\delta_s)_{s>0}$, and every point $p\in\Hek$ can be uniquely written as $p=p_{\V_\nu} \cdot p_{\LL_\nu}$ for some $p_{\V_\nu}\in \V_\nu$ and $p_{\LL\nu}\in \LL_\nu$. We define the projections onto $\V_\nu$ and $\LL_\nu$ by
\begin{align*}
\pi_{\V_\nu}:\Hek \to \V_\nu,\quad &\pi_{\V_\nu}(p)=p_{\V_\nu},\\
\pi_{\LL_\nu}:\Hek \to \LL_\nu,\quad &\pi_{\LL_\nu}(p)=p_{\LL_\nu}.
\end{align*}

We note for further use that these projections are continuous and that they commutes with the dilations $\delta_s$. This can for instance be seen from the following explicit expressions. For $p=(v,t)\in\Hek$, 
\begin{equation} \label{e:plnu}
\pi_{\V_\nu}(p) = \left(v-\langle v,\nu \rangle \nu, t- \omega(v,\langle v,\nu \rangle \nu)/2\right) \, \text{and }  \pi_{\LL_\nu}(p) = (\langle v,\nu \rangle \nu,0)
\end{equation}
where $\langle \cdot ,\cdot \rangle$ denotes the Euclidean scalar product in $\Rk$.

\subsection{Intrinsic Lipschitz graphs} \label{subsect:intrinsic-Lip-graphs}
Given $\lambda >0$ and $\nu \in \Sk$, we set
\begin{equation*} 
C_\lambda(\nu) = \left\{p\in\Hek : \lambda\|p\|<\|\pi_{\LL_\nu}(p)\|\right\}~.
\end{equation*}

These are open subsets of $\Hek\setminus\{0\}$. Since the projection $\pi_{\LL_\nu}$ commutes with the dilations and by homogeneity of the Kor\'anyi norm, these sets are cones. Recall that we say that a set $A \subset \Hek$ is a cone if $\delta_s(A) \subset A$ for each $s>0$. For $\lambda < \lambda'$ we have $\overline{C_{\lambda'}(\nu)} \setminus \{0\} \subset C_\lambda(\nu)$ and
\begin{equation} \label{e:intersection-cones}
\bigcap_{\lambda\in (0,1)} C_\lambda(\nu)  = \LL_\nu \setminus \{0\}~.
\end{equation}
Indeed it follows from the definition of $C_\lambda(\nu)$ that $\cap_{\lambda\in (0,1)} C_\lambda(\nu) = \{p\in \Hek\setminus\{0\}: \|p\| \leq \|\pi_{\LL_\nu}(p)\|\}$. Remembering the expression of the Kor\'anyi norm, we get from~\eqref{e:plnu} that $\|\pi_{\LL_\nu}(p)\| \leq \|p\|$ for every $p\in \Hek$. Hence $\cap_{\lambda\in (0,1)} C_\lambda(\nu) = \left\{p\in \Hek\setminus\{0\}: \|p\| = \|\pi_{\LL_\nu}(p)\|\right\}$. Next, once again by~\eqref{e:plnu} and the expression of the Kor\'anyi norm, it can easily be seen that $\|p\| = \|\pi_{\LL_\nu}(p)\|$ if and only if $p=\pi_{\LL_\nu}(p)$, that is, $p\in\LL_\nu$, and~\eqref{e:intersection-cones} follows. Hence $(C_\lambda(\nu))_{\lambda\in (0,1)}$ is a family of open cones in $\Hek\setminus\{0\}$ that shrinks to $\LL_\nu\setminus\{0\}$ and Definition~\ref{def:lip-graphs} mimicks one of the various geometric characterizations of codimension one Lipschitz graphs in $\R^n$ equipped with its Euclidean structure and Euclidean norm.

\begin{definition} \label{def:lip-graphs} 
We say that $\Gamma \subset \Hek$ is an intrinsic Lipschitz graph with codimension one if there are $\lambda\in (0,1)$ and $\nu\in\Sk$ such that $\pi_{\V_\nu} (\Gamma) = \V_\nu$ and $\left(p\cdot C_{\lambda}(\nu)\right)\cap \Gamma =\emptyset$ for all $p\in \Gamma$.
\end{definition}

Since we will be only concerned with intrinsic Lipschitz graph with codimension one, we say in the rest of this paper that a subset of $\Hek$ is an intrinsic Lipschitz graph to mean that it is an intrinsic Lipschitz graph with codimension one. More precisely we say that $\Gamma \subset \Hek$ is an intrinsic $(\lambda,\nu)$-Lipschitz graph if $\lambda \in (0,1)$ and $\nu\in\Sk$ are such that the two conditions in Definition~\ref{def:lip-graphs} hold. Given $\lambda \in (0,1)$ we say that $\Gamma \subset \Hek$ is an intrinsic $\lambda$-Lipschitz graph if it is an intrinsic $(\lambda,\nu)$-Lipschitz graph for some $\nu\in\Sk$.


\begin{remark} \label{rk:lip-graphs-are-intrinsic-graphs} We note that any subset of an intrinsic $(\lambda,\nu)$-Lipschitz graph is an intrinsic graph. Indeed let $\lambda\in (0,1)$, $\nu\in\Sk$, and $A\subset \Hek$ be such that $(p\cdot C_{\lambda}(\nu))\cap A =\emptyset$ for all $p\in A$. Then there is a map $\varphi:\pi_{\V_\nu}(A)\rightarrow\LL_\nu$ such that $A= \{ p\cdot \varphi(p): p\in \pi_{\V_\nu}(A)\}$. Indeed let $p$, $q\in A$ be such that $\pi_{\V_\nu}(p)=\pi_{\V_\nu}(q)$. Then $p^{-1}\cdot q \notin C_\lambda(\nu)$ with 
$$p^{-1}\cdot q = \pi_{\LL_\nu}(p)^{-1}\cdot \pi_{\V_\nu}(p)^{-1} \cdot \pi_{\V_\nu}(q) \cdot \pi_{\LL_\nu}(q) = \pi_{\LL_\nu}(p)^{-1}\cdot \pi_{\LL_\nu}(q)~.$$ 
Hence $p^{-1}\cdot q \in \LL_\nu \setminus C_\lambda(\nu) = \{0\}$, that is, $p=q$. It follows in particular that an intrinsic $(\lambda,\nu)$-Lipschitz graph $\Gamma$ can be written as 
\begin{equation*}
\Gamma = \left\{ p\cdot \varphi_\Gamma(p): p\in \V_\nu \right\}
\end{equation*}
for some map $\varphi_\Gamma:\V_\nu\rightarrow\LL_\nu$. We stress that in contrast to the Euclidean setting the map $\varphi_\Gamma$ seen as a map between the metric spaces $\V_\nu$ and $\LL_\nu$ endowed with the restriction of the Kor\'anyi distance might not be Lipschitz, and is instead locally $1/2$-H\"older continous, see~\cite[Proposition~3.8]{MR3511465} and Remark~\ref{rk:FSSC-NY-lip-graphs}.
\end{remark}

\begin{remark} We stress that the condition $\pi_{\V_\nu} (\Gamma) = \V_\nu$ in Definition~\ref{def:lip-graphs}, that is, considering global intrinsic Lipschitz graphs over the whole $\V_\nu$, is not a restrictive one. We indeed recall the following sharp extension property:

\begin{theorem}\cite{MR3815462} \label{thm:sharp-extension-lip-graphs}
Let $\lambda\in (0,1)$, $\nu\in\Sk$, and $A\subset \Hek$ be such that $(p\cdot C_{\lambda}(\nu))\cap A =\emptyset$ for all $p\in A$. Then there is an intrinsic $(\lambda,\nu)$-Lipschitz graph $\Gamma$ such that $A \subset \Gamma$.
\end{theorem}

This sharp extension property is proved in~\cite[Theorem~27]{MR3815462} with a sligthly different definition of the cones $C_\lambda(\nu)$ where the Kor\'anyi norm is replaced by the Carnot-Carath\'eodory one. However the same arguments apply when working with the Kor\'anyi norm, and more generally with any other homogeneous norm, as can been seen from the proof given in~\cite{MR3815462}.
\end{remark}

The following notion of equivalent families of subsets of $\Hek$ will be useful to get equivalent characterizations of intrinsic Lipschitz graphs. Given sets $I$, $J$ of indices, we say that two families $(A_\alpha)_{\alpha \in I}$ and $(B_\beta)_{\beta \in J}$ of subsets of $\Hek$ are equivalent if for each $\alpha \in I$ there is $\beta \in J$ such that $B_\beta \subset A_\alpha$ and if for each $\beta \in J$ there is $\alpha \in I$ such that $A_\alpha \subset B_\beta$.

Given an interval $I \subset \R$ we say that a family $(A_\alpha)_{\alpha \in I}$ of subsets of $\Hek\setminus\{0\}$ is increasing if for all $\alpha$,  $\alpha' \in I$ such that $\alpha < \alpha'$ one has $\overline{A_\alpha} \setminus\{0\} \subset A_{\alpha'}$. Similarly we say that $(A_\alpha)_{\alpha \in I}$ is decreasing if for all $\alpha$,  $\alpha' \in I$ such that $\alpha < \alpha'$ one has $\overline{A_{\alpha'}} \setminus\{0\} \subset A_{\alpha}$. We say that $(A_\alpha)_{\alpha \in I}$ is monotone if it is either increasing or decreasing. 

\begin{proposition} \label{prop:equivalent-families} Let $I$, $J \subset \R$ be open intervals. Let $(A_\alpha)_{\alpha \in I}$ and $(B_\beta)_{\beta \in J}$ be two monotone families of open cones in $\Hek\setminus\{0\}$. Assume that 
\begin{equation*} \label{e:hyp-equivalent-families-cones}
\bigcap_{\alpha \in I} A_\alpha = \bigcap_{\beta \in J} B_\beta~.
\end{equation*}
Then $(A_\alpha)_{\alpha \in I}$ and $(B_\beta)_{\beta \in J}$ are equivalent.
\end{proposition}

\begin{proof}
Assume that the family $(A_\alpha)_{\alpha \in I}$ is increasing, the case where $(A_\alpha)_{\alpha \in I}$ is decreasing being similar. Let $\beta\in J$ be given and let us prove that there is $\alpha \in I$ such that $A_\alpha \subset B_\beta$. We first prove that one can find $\alpha \in I$ such that 
\begin{equation} \label{e:equivalent-cones}
A_\alpha \cap \Se \subset B_\beta
\end{equation}
where $\Se = \{p\in\Hek : \|p\|=1\}$.

Arguing by contradiction assume that for each $\alpha \in I$ one can find $p_\alpha \in (A_\alpha \cap \Se) \setminus B_\beta$. Since $\Se$ is compact, one can find a sequence $\alpha_l \downarrow\inf I$  such that $p_{\alpha_l}$ converges to some $p\in \Se$. Since $B_\beta$ is open, $p\notin B_\beta$. 

On the other hand let $\alpha$, $\alpha' \in I$ be such that $\inf I <\alpha'<\alpha$. By monotonicity, we have $p_{\alpha_l} \in A_{\alpha_l} \subset A_{\alpha'}$ for all $l$ large enough. Hence $p\in \overline{A_{\alpha'}} \cap \Se \subset A_\alpha $. Since this holds for each $\alpha\in I$, it follows that $p\in \cap_{\alpha \in I} A_\alpha = \cap_{\beta' \in J} B_{\beta'} \subset B_\beta$. This gives a contradiction and concludes the proof of~\eqref{e:equivalent-cones}.

It follows from~\eqref{e:equivalent-cones} and the fact that the sets $A_\alpha$ and $B_\beta$ are cones in $\Hek\setminus\{0\}$ that
\begin{equation*}
A_\alpha = \bigcup_{s>0} \delta_s(A_\alpha \cap \Se) \subset \bigcup_{s>0} \delta_s(B_\beta \cap \Se) = B_\beta
\end{equation*}
which concludes the proof of the proposition.
\end{proof}

Given $\beta>0$ and $\nu\in\Sk$, we set

\begin{equation*} 
 D_\beta(\nu) = \bigcup_{p\in\LL_\nu\setminus\{0\}} B(p,\beta \|p\|)~.
\end{equation*}

These sets have already been considered in~\cite[Remark~A.2]{MR3906289}, with a slightly different definition though. For our purposes the following uniform equivalence between the families $(C_{\lambda}(\nu))_{\lambda\in (0,1)}$ and $(D_{\beta}(\nu))_{\beta\in(0,1)}$ will be useful.

\begin{proposition} \label{prop:Dbeta-lip-graph}
For every $\lambda \in (0,1)$ there is $\beta\in(0,1)$ such that for every $\nu \in \Sk$ we have $D_{\beta}(\nu)\subset C_{\lambda}(\nu)$. Conversely for every $\beta \in (0,1)$ there is $\lambda\in(0,1)$ such that for every $\nu \in \Sk$ we have $C_{\lambda}(\nu) \subset D_{\beta}(\nu)$.
\end{proposition}

\begin{proof}
We first prove that for each $\nu \in \Sk$ the families $(C_{\lambda}(\nu))_{\lambda\in (0,1)}$ and $(D_{\beta}(\nu))_{\beta\in(0,1)}$ are equivalent. We already know that $(C_{\lambda}(\nu))_{\lambda\in (0,1)}$ is a decreasing family of open cones in $\Hek\setminus\{0\}$ that shrinks to $\LL_\nu\setminus\{0\}$. It can easily be checked that $(D_{\beta}(\nu))_{\beta\in(0,1)}$ is an increasing family of open cones in $\Hek\setminus\{0\}$. We obviously have $\LL_\nu\setminus\{0\} \subset \cap_{\beta\in(0,1)}D_{\beta}(\nu)$. Conversely let $q \in \cap_{\beta\in(0,1)}D_{\beta}(\nu)$. For each $\beta \in (0,1)$ there is $p_\beta \in \LL_\nu\setminus\{0\}$ such that 
\begin{equation} \label{e:Dbeta-lip-graph}
d(q,p_\beta) < \beta d(0,p_\beta)~.
\end{equation}
It follows that $\|p_\beta\| < (1-\beta)^{-1} \|q\|$. Hence $(p_\beta)_{\beta\in(0,1)}$ is bounded and one can find a sequence $\beta_l \downarrow 0$ such that $p_{\beta_l}$ converges to some $p\in \LL_\nu$. Going back to~\eqref{e:Dbeta-lip-graph} we get that $q=p$. Since $0\notin \cap_{\beta\in(0,1)} D_{\beta}(\nu)$, it follows that $q\in \LL_\nu \setminus \{0\}$. Hence
\begin{equation*}
\bigcap_{\beta\in(0,1)}D_{\beta}(\nu) = \LL_\nu \setminus \{0\} = \bigcap_{\lambda\in (0,1)} C_\lambda(\nu)
\end{equation*}
and we get that $(C_{\lambda}(\nu))_{\lambda\in (0,1)}$ and $(D_{\beta}(\nu))_{\beta\in(0,1)}$ are equivalent from Proposition~\ref{prop:equivalent-families}.

To conclude the proof, let $\nu$, $\nu' \in \Sk$ be given. Let $\rho:\R^{2k}\rightarrow\R^{2k}$ be a linear map that preserves the Euclidean scalar product and the simplectic form $\omega$ and such that $\rho(\nu') = \nu$. The existence of such a map follows from elementary linear algebra. Then the map $i:(\Hek,d)\rightarrow(\Hek,d)$ defined by $i(v,t) = (\rho(v),t)$ is an isometry. Moreover $\pi_{\LL_{\nu}} \circ i = i \circ \pi_{\LL_{\nu'}}$ and $\pi_{\V_{\nu}} \circ i = i \circ \pi_{\V_{\nu'}}$. It follows that for $\lambda$, $\beta\in (0,1)$, one has $i(C_{\lambda}(\nu')) = C_{\lambda}(\nu)$ and $i(D_{\beta}(\nu')) = D_{\beta}(\nu)$.
\end{proof}

\begin{remark} \label{rk:FSSC-NY-lip-graphs} Intrinsic Lipschitz graphs were introduced by B.~Franchi, R.~Serapioni, and F.~Serra Cassano~\cite{MR2287539} using the family of cones
\begin{equation*} 
\widetilde{C}_{\gamma}(\nu) = \left\{p \in \Hek:  \|\pi_{\V_\nu}(p)\|_\infty < \gamma \|\pi_{\LL_\nu}(p)\|_\infty \right\} 
\end{equation*}
where $\|(v,t)\|_\infty = \max\left\{|v|, 2\sqrt{|t|}\right\}$. A notion of intrinsic Lipschitz graphs was also considered in~\cite[Section~2.3]{MR3815462} using the family of cones
\begin{equation*} 
\cone_{\lambda}(\nu) = \left\{p\in \Hek: \lambda d_{cc}(0,p) < d_{cc}(0,\pi_{\LL_\nu}(p))\right\}
\end{equation*}
where $d_{cc}$ denotes the Carnot-Carath\'eodory distance on $\Hek$. It can easily be seen from the definition that $(\widetilde{C}_{\gamma}(\nu))_{\gamma>0}$ is an increasing family of open cones in $\Hek\setminus\{0\}$ with $$\bigcap_{\gamma>0} \widetilde{C}_{\gamma}(\nu)=\LL_\nu\setminus\{0\}~.$$ It can also easily be seen that $(\cone_{\lambda}(\nu))_{\lambda\in (0,1)}$ is an decreasing family of open cones in $\Hek\setminus\{0\}$. To show that this family shrinks to $\LL_\nu\setminus\{0\}$, we note that the Carnot-Carath\'eodory distance shares with the Kor\'anyi one the following two properties. First $d_{cc}(0,\pi_{\LL_\nu}(p)) \leq d_{cc}(0,p)$ for every $p\in \Hek$. Second $d_{cc}(0,p) = d_{cc}(0,\pi_{\LL_\nu}(p))$ if and only if $p\in\LL_\nu$. From these we get 
 \begin{equation*} 
\bigcap_{\lambda\in (0,1)} \cone_{\lambda}(\nu) = \LL_\nu \setminus \{0\}~.
\end{equation*}
By Proposition~\ref{prop:equivalent-families} $(\widetilde{C}_{\gamma}(\nu))_{\gamma>0}$, $(\cone_{\lambda}(\nu))_{\lambda\in (0,1)}$, and $(C_{\lambda}(\nu))_{\lambda\in (0,1)}$ are equivalent families. On the one hand this implies that the classes of intrinsic Lipschitz graphs introduced in~\cite[Definition~4.54]{MR3587666} and in~\cite[Section~2.3]{MR3815462} coincide. On the other hand, taking into account Remark~\ref{rk:lip-graphs-are-intrinsic-graphs} and Theorem~\ref{thm:sharp-extension-lip-graphs}, we get that this class of intrinsic Lipschitz graphs originally introduced in the literature consists of all subsets of intrinsic Lipschitz graphs in the sense of Definition~\ref{def:lip-graphs}.
\end{remark}

\subsection{Intrinsic Lipschitz graphs are Semmes surfaces} \label{subsect:intrinsic-Lip-graphs-are-Semmes-surfaces}

Semmes surfaces have been introduced in the Euclidean setting~\cite{MR948198} as codimension one Ahlfors regular sets that satisfy an additional condition called condition~$B$. Such a metric notion can be naturally extended in the Heisenberg setting, see~\cite{FOR-arxiv} for more details.

\begin{definition} \label{def:semmes-surface} We say that $S\subset \Hek$ is a Semmes surface if $S$ is Ahlfors regular and there is $c>0$ such that $S$ satisfies condition~$B$ with constant $c$, meaning that each ball centered on $S$ with radius $r>0$ contains two balls with radius $cr$ that are contained in different connected components of $S^c$.
\end{definition}

Propositions~\ref{prop:AR-intrinsic-Lip-graphs} and~\ref{prop:conditionB-intrinsic-Lip-graphs} below imply that intrinsic Lipschitz graphs are Semmes surfaces, and, more importantly for our purposes, this property comes with bounds on their Ahlfors regularity and condition~$B$ constants.

\begin{proposition} \label{prop:AR-intrinsic-Lip-graphs}
For every $\lambda \in (0,1)$ there is $C \geq 1$ such that every intrinsic $\lambda$-Lipschitz graph is $C$-Ahlfors regular.
\end{proposition}

\begin{proof}
We first note that an intrinsic Lipschitz graph $\Gamma$ is a closed subset of $\Hek$. This can be deduced from the continuity of the map $\varphi_\Gamma$ given by Remark~\ref{rk:lip-graphs-are-intrinsic-graphs} together with the continuity of $\pi_{\V_\nu}$ and $\pi_{\LL_{\nu}}$. Next let $\lambda \in (0,1)$ be given and fix $\nu\in\Sk$. By~\cite[Theorem~3.9]{MR3511465} and Remark~\ref{rk:FSSC-NY-lip-graphs} there is $C>0$ such that every intrinsic $(\lambda,\nu)$-Lipschitz graph is $C$-Ahlfors regular. Let $\nu'\in\Sk$ and let $i:(\Hek,d)\rightarrow(\Hek,d)$ be the isometry considered at the end of the proof of Proposition~\ref{prop:Dbeta-lip-graph}. Then, if $\Gamma$ is an intrinsic $(\lambda,\nu')$-Lipschitz graph, $i(\Gamma)$ is a $(\lambda,\nu)$-Lipschitz graph, and the $C$-Ahlfors regularity of $\Gamma$ follows from the $C$-Ahlfors regularity of $i(\Gamma)$.
\end{proof} 

\begin{proposition} \label{prop:conditionB-intrinsic-Lip-graphs}
For every $\lambda \in (0,1)$ there is $c >0$ such that every intrinsic $\lambda$-Lipschitz graph $\Gamma$ satisfies condition~$B$ with constant $c$.
\end{proposition}

\begin{proof}
Let $\lambda \in (0,1)$ be given and $\Gamma$ be an intrinsic $(\lambda,\nu)$-Lipschitz graph for some $\nu\in\Sk$. By Proposition~\ref{prop:Dbeta-lip-graph} there is $\beta\in (0,1)$, depending only on $\lambda$, such that $\left(p\cdot D_{\beta}(\nu)\right)\cap \Gamma =\emptyset$ for all $p\in \Gamma$. Let $p\in\Gamma$ and $r>0$ be given. Set $q_1 =(- 2^{-1} r \nu,0) \in \LL_\nu \setminus\{0\}$ and $q_2=(2^{-1}r \nu,0)\in \LL_\nu \setminus\{0\}$. Next, for $i=1,2$, set $p_i= p\cdot q^i$. We have
\begin{equation*}
B(p_i,\beta r/2) = p\cdot B(q_i,\beta \|q_i\|) \subset B(p,r) \cap (p\cdot D_\beta(\nu)) \subset B(p,r) \setminus \Gamma~.
\end{equation*}
To prove that $\Gamma$ satisfies condition~$B$ with constant $c=\beta/2$, it remains to check that $p_1$ and $p_2$ belong to different connected components of $\Gamma^c$. Let $\sigma:[0,1]\rightarrow \Hek$ be a continuous path joining $p_1=\sigma(0)$ to $p_2=\sigma(1)$. By Remark~\ref{rk:lip-graphs-are-intrinsic-graphs}, we can write $\Gamma$ as $\Gamma = \left\{ q\cdot \varphi_\Gamma(q): q\in\V_\nu\right\}$ for some continous map $\varphi_\Gamma:\V_\nu \rightarrow \LL_\nu$. Define $h:[0,1]\rightarrow \LL_\nu$ by $h(t)= \varphi_\Gamma(\pi_{\V_\nu}(\sigma(t))^{-1}\cdot \pi_{\LL_\nu} (\sigma(t))$. For $i=1,2$, we have $p_i = p\cdot q_i = \pi_{\V_\nu}(p) \cdot \varphi_\Gamma(\pi_{\V_\nu}(p)) \cdot q_i$. Hence $\pi_{\V_\nu}(p_i) =  \pi_{\V_\nu}(p)$ and $\pi_{\LL_\nu}(p_i) = \varphi_\Gamma(\pi_{\V_\nu}(p)) \cdot q_i$. It follows that $h(0)= q_1 =(-2^{-1}r \nu,0)$ and $h(1)=q_2=(2^{-1} r \nu,0)$. By continuity of $h$ there is $t\in(0,1)$ such that $h(t)=0$, that is, $\pi_{\LL_\nu} (\sigma(t)) = \varphi_\Gamma(\pi_{\V_\nu}(\sigma(t))$. Therefore $\sigma(t) \in \Gamma$. Hence every continuous path from $p_1$ to $p_2$ meets $\Gamma$ and this implies that $p_1$ and $p_2$ belong to different connected components of $\Gamma^c$.
\end{proof}

\subsection{Big pieces of intrinsic Lipschitz graphs} \label{subsect:BPiLG}

We end this section with the definition of the big pieces of intrinsic Lipschitz graphs condition.

\begin{definition} \label{def:BPiLG} We say that a set $E\subset\Hek$ has big pieces of intrinsic Lipschitz graphs (BPiLG) is $E$ if Ahlfors regular and there are $\lambda\in(0,1)$ and $\theta>0$ such that for each $p\in E$, $r>0$, there is an intrinsic $\lambda$-Lipschitz graph $\Gamma$ such that 
\begin{equation*}
\calH^{2k+1} (E\cap \Gamma \cap B(p,r)) \geq \theta r^{2k+1}~.
\end{equation*}
\end{definition} 

Given $C\geq 1$, $\lambda \in (0,1)$, and $\theta >0$, we denote by $\BPiLG(C,\lambda,\theta)$ the class of $C$-Ahlfors regular subsets of $\Hek$ for which the condition given in Definition~\ref{def:BPiLG} holds for the given values of the parameters $\lambda$ and $\theta$.


\section{Bilateral weak geometric lemma} \label{sect:BWGL}

\subsection{Definitions} \label{subsect:def-bwgl}

We first recall the definition of Carleson sets. Let $\gamma >0$ and $E\subset \Hek$ be Ahlfors regular. A $\gamma$-Carleson set $A$ is a measurable subset of $E\times \R^+$ such that 
\begin{equation} \label{e:def-Carleson-set}
\int_0^r \int_{E\cap B(p,r)} \chi_A(q,s) \, d\calHk (q) \,\frac{ds}{s} \leq \gamma r^{2k+1} \quad \text{for all } p\in E, \, r>0~.
\end{equation}

Next we recall the definition of the bilateral $\beta$-numbers. For the sake of completeness we give the definition of bilateral $\beta$-numbers both with respect to arbitrary and vertical hyperplanes. Let $\calP$ denote the set of all affine hyperplanes in $\Hek$ identified with $\R^{2k+1}$ as a real vector space and $\calV$ denote the subset of $\calP$ consisting of all vertical hyperplanes, that is, sets of the form $p\cdot \V_\nu$ for some $p\in \Hek$ and some $\nu\in\V_\nu$. Given $E\subset\Hek$, $p\in E$, and $s>0$, we define the bilateral $\beta$-numbers $b\beta_{E} (p,s)$ and $b\beta_{v,E} (p,s)$ by
\begin{equation} \label{e:def-bbeta}
b\beta_{E}(p,s) = s^{-1} \,\inf_{P \in \calP}\,  \left\{ \sup_{q \in B(p,s) \cap E} \dist(q,P) + \sup_{q \in B(p,s) \cap P} \dist(q,E) \right\}~,
\end{equation}
\begin{equation*}
b\beta_{v,E}(p,s) = s^{-1} \,\inf_{V \in \calV}\,  \left\{ \sup_{q \in B(p,s) \cap E} \dist(q,V) + \sup_{q \in B(p,s) \cap V} \dist(q,E) \right\}.
\end{equation*}

It turns out that for our purposes considering either of theses versions of the bilateral $\beta$-numbers does not matter. The following result is indeed implicitly contained in~\cite[Section~9.4]{MR3815462}, see in~\cite[Theorem~5.10]{FOR-arxiv}.

\begin{theorem} \cite[Section~9.4]{MR3815462},~\cite[Theorem~5.10]{FOR-arxiv} \label{thm:BWGL-vertical-versus-arbitrary}
Let $C\geq 1$ and $E\subset \Hek$ be $C$-Ahlfors regular. Then the following conditions are equivalent:
\begin{itemize}

\item[$(i)$] there is $\gamma_1:(0,1)\rightarrow(0,+\infty)$ such that for each $\epsilon\in(0,1)$ the set $\left\{(p,s)\in E \times (0,+\infty) : b\beta_{E}(p,s) > \epsilon \right\}$ is $\gamma_1(\epsilon)$-Carleson,
\smallskip
\item[$(ii)$] there is $\gamma_2:(0,1)\rightarrow(0,+\infty)$ such that for each $\epsilon\in(0,1)$ the set $\left\{(p,s)\in E \times (0,+\infty) : b\beta_{v,E}(p,s) > \epsilon \right\}$ is $\gamma_2(\epsilon)$-Carleson.
\end{itemize}
Moreover, if~$(i)$, respectively~$(ii)$, holds for some $\gamma_1$, respectively $\gamma_2$, then $(ii)$, respectively~$(i)$, holds for some $\gamma_2$, respectively $\gamma_1$, that can be chosen depending only on $C$ and $\gamma_1$, respectively $\gamma_2$. 
\end{theorem}

\begin{definition}[Bilateral weak geometric lemma] \label{def:bwgl}
We say that $E\subset\Hek$ satisfies the \textit{bilateral weak geometric lemma} (BWGL) if $E$ is Ahlfors regular and one of the equivalent conditions in Theorem~\ref{thm:BWGL-vertical-versus-arbitrary} holds. 
\end{definition}

Given $C \geq 1$ and $\gamma:(0,1) \rightarrow (0,+\infty)$ we denote by $\BWGL(C,\gamma)$ the class of $C$-Ahlfors regular sets $E\subset\Hek$ such that for each $\epsilon \in (0,1)$ the set $\left\{(p,s)\in E \times (0,+\infty) : b\beta_{E}(p,s) > \epsilon \right\}$ is $\gamma(\epsilon)$-Carleson.

\subsection{BWGL for intrinsic Lipschitz graphs} \label{subsect:bwgl-for-intrinsic-Lip-graphs}.

The validity of the bilateral weak geometric lemma for intrinsic Lipschitz graphs, with suitable bounds, is the first step in the proof Theorem~\ref{thm:main}. More precisely the following result follows from Propositions~\ref{prop:AR-intrinsic-Lip-graphs} and~\ref{prop:conditionB-intrinsic-Lip-graphs} together with~\cite[Proposition~5.11]{FOR-arxiv}. 

\begin{theorem} \label{thm:bwgl-for-intrinsic-Lip-graphs} For every $\lambda \in (0,1)$ there are $C\geq 1$ and $\gamma:(0,1) \rightarrow (0,+\infty)$ such that every intrinsic $\lambda$-Lipschitz graph belongs to $\BWGL(C,\gamma)$.
\end{theorem}

\subsection{Stability of BWGL} \label{subsect:stability-bwgl}

The second step in the proof of Theorem~\ref{thm:main} is a stability result for BWGL, Theorem~\ref{thm:stability-bwgl}, proved in this section. It is inspired by~\cite[Part~IV-Chapter~1]{MR1251061} where the stability for the Euclidean version of the weak geometric lemma, a weaker version than the bilateral one, is proved.

\begin{theorem} \label{thm:stability-bwgl}
Let $C \geq 1$ and $E\subset\Hek$ be $C$-Ahlfors regular. Assume that there are $\theta>0$, $C_1 \geq 1$, and $\gamma_1:(0,1)\rightarrow (0,+\infty)$ such that for each $p\in E$, $r>0$, there is $\widetilde{E} \in \BWGL(C_1,\gamma_1)$ which satisfies
\begin{equation*}
\calHk (E \cap \widetilde{E} \cap B(p,r)) \geq \theta r^{2k+1}~.
\end{equation*}
Then $E \in \BWGL(C,\gamma)$ for some $\gamma:(0,1)\rightarrow (0,+\infty)$ depending only on $C$, $\theta$, $C_1$ and $\gamma_1$.
\end{theorem}

The rest of this section is devoted to the proof of Theorem~\ref{thm:stability-bwgl}. Discrete versions of Carleson packing conditions for systems of dyadic cubes on Ahlfors regular sets will be technically useful for our purposes. We first recall these rather standard notions. 

\begin{definition} \label{def:dyadic-cubes} Let $C\geq 1$ and $E\subset\Hek$ be $C$-Ahlfors regular. We say that $\Delta = \cup_{j\in\Z} \Delta_j$ is a system of dyadic cubes on $E$ if the following three conditions hold:
\begin{itemize}
\item[$(i)$] for each $j\in \Z$, $\Delta_j$ is a family of disjoint subsets of $E$ such that $\calHk(E\setminus\cup_{Q\in\Delta_j} Q) = 0$,
\smallskip
\item[$(ii)$] if $Q\in \Delta_j$ and $Q'\in \Delta_l$ for some $j\leq l$ then either $Q\cap Q' = \emptyset$ or $Q\subseteq Q'$,
\smallskip
\item[$(iii)$] there is $C_0 \geq 1$ depending only on $C$ such that for all $j\in\Z$ and all $Q\in \Delta_j$ there is a ball $B$ centered on $Q$ with radius $C_0^{-1} 2^j$ such that $B\cap E \subset Q$ and $\diam Q \leq C_0 2^j$.
\end{itemize}
\end{definition}

The existence of systems of dyadic cubes on Ahlfors regular subsets of the Euclidean space is due to G.~David~\cite{MR1009120,MR1123480}. The construction has been extended by M.~Christ~\cite{MR1096400} to the general metric setting. 

Given $C\geq 1$,  a $C$-Ahlfors-regular set $E\subset\Hek$, a system of dyadic cubes $\Delta$ on $E$, $j\in\Z$, and $Q\in\Delta_j$,  we set $\ell(Q)=2^j$ and $|Q|=2^{j(2k+1)}$. We also fix a point $p_Q \in Q$ and we set $B(Q) = B(p_Q, 2C_0 \ell(Q))$ where $C_0$ is the constant that shows up in Definition~\ref{def:dyadic-cubes}~$(iii)$. Note that Definition~\ref{def:dyadic-cubes}~$(iii)$ implies that there is $M\geq 1$ depending only on $C$ such that $M^{-1} |Q| \leq \calHk(Q) \leq M |Q|$.

Given $\gamma >0 $ we say that a subset $\mathcal{A}$ of $\Delta$ satisfies a $\gamma$-Carleson packing condition if
\begin{equation*}
\sum_{\substack{Q \subseteq R \\ Q\in\mathcal{A}}} |Q| \leq \gamma |R|  \qquad \text{for all }  R \in \Delta~.
\end{equation*}

We recall now a well-known relationship between Carleson sets and Carleson packing conditions for systems of dyadic cubes. More general statements could be given but Lemma~\ref{lem:carleson-packing-conditions-continuous-vs-discrete} will be sufficient for our purposes.

\begin{lemma} \label{lem:carleson-packing-conditions-continuous-vs-discrete} 
Let $C\geq 1$, $E\subset\Hek$ be $C$-Ahlfors regular, $\Delta$ be a system of dyadic cubes on $E$, and $f:E\times (0,+\infty) \rightarrow \R^+$ be given. For $Q\in\Delta$ set $f(Q) = f(p_Q,2C_0 \ell(Q))$. Assume that
\begin{gather} 
f(Q) \leq 4 f(q,s) \quad \text{for all } Q\in\Delta,\,\, q\in Q,\,\, 4C_0 \ell(Q) < s \leq  8C_0 \ell(Q)~, \label{e:f(Q)-vs-f(q,s)}\\
f(q,s) \leq 4 f(Q) \quad \text{for all } Q\in\Delta,\,\, q\in Q,\,\, \frac{C_0}{2} \ell(Q) \leq s < C_0 \ell(Q)~. \label{e:f(q,s)-vs-f(Q)}
\end{gather}
Then the following conditions are equivalent:
\begin{itemize}
\item[$(i)$] there is $\gamma_1:(0,1)\rightarrow(0,+\infty)$ such that for each $\epsilon\in(0,1)$ the set $\left\{(p,s)\in E \times (0,+\infty) : f(p,s) > \epsilon \right\}$ is $\gamma_1(\epsilon)$-Carleson,
\smallskip
\item[$(ii)$] there is $\gamma_2:(0,1)\rightarrow(0,+\infty)$ such that for each $\epsilon\in(0,1)$ the set $\left\{Q\in\Delta : f(Q) >\epsilon\right\}$ satisfies a $\gamma_2(\epsilon)$-Carleson packing condition.
\end{itemize}
Moreover, if~$(i)$, respectively~$(ii)$, holds for some $\gamma_1$, respectively $\gamma_2$, then $(ii)$, respectively~$(i)$, holds for some $\gamma_2$, respectively $\gamma_1$, that can be chosen depending only on $C$ and $\gamma_1$, respectively $\gamma_2$. 
\end{lemma}

Given $E_1$, $E_2\subset\Hek$, $p\in\Hek$, $s>0$, we set
\begin{equation*}
I_{E_1,E_2}(p,s)=s^{-1} \sup_{\substack{q\in  B(p,s) \cap E_1 \\ \dist (q,E_2)< s}} \dist (q,E_2)
\end{equation*}
where we consider the supremum over the empty set to be zero.

\begin{lemma} \label{lem:I-carleson-sets} For every $C\geq 1$ there is $\gamma:(0,1)\rightarrow (0,+\infty)$ such that if $E_1\subset\Hek$ is $C$-Ahlfors regular and $E_2\subset\Hek$ then for each $\epsilon\in(0,1)$ the set $\{(p,s)\in E_1\times (0,+\infty) : I_{E_1,E_2}(p,s) >\epsilon\}$ is $\gamma(\epsilon)$-Carleson.
\end{lemma}

\begin{proof} We refer to~\cite[Part~IV-Lemma~1.42]{MR1251061} for the Euclidean analog of Lemma~\ref{lem:I-carleson-sets} whose proof can be translated to our setting as we explain now for the reader convenience. Let $\Delta$ be a system of dyadic cubes on $E_1$. One can easily check that~\eqref{e:f(Q)-vs-f(q,s)} and~\eqref{e:f(q,s)-vs-f(Q)} are satisfied by $f=I_{E_1,E_2}$. Hence Lemma~\ref{lem:carleson-packing-conditions-continuous-vs-discrete} applies and we are going to prove the existence of $\gamma:(0,1)\rightarrow (0,+\infty)$ depending only on $C$ such that for each $\epsilon\in(0,1)$ the set $\mathcal{B}_\epsilon = \left\{Q\in\Delta : I_{E_1,E_2}(Q) >\epsilon\right\}$ satisfies a $\gamma(\epsilon)$-Carleson packing condition. In the rest of this proof we write $A\lesssim B$ to mean that there is $M\geq 0$ whose value depends only on $C$ and $\epsilon$ such that $A\leq M B$. 

Let $\epsilon \in (0,1)$ be given. We have
\begin{equation*}
\mathcal{B}_\epsilon =\left\{Q\in\Delta : \exists\, q \in B(Q)\cap E_1, \, 2C_0\ell(Q)\epsilon < \dist(q,E_2) < 2C_0\ell(Q) \right\}~.
\end{equation*}

For each $j\in\Z$, define $l(j) \in \Z$ as the unique integer such that $2^{l(j)} \leq 2^j\epsilon < 2^{l(j)+1}$. Note that the map $j\in\Z \mapsto l(j) \in \Z$ is a bijection. To each $Q\in \mathcal{B}_\epsilon$ we associate $T(Q)\in\Delta$ in the following way. For $j\in \Z$ and $Q\in \mathcal{B}_\epsilon\cap\Delta_j$ we choose a point $q \in B(Q)\cap E_1$ such that $2C_0\ell(Q)\epsilon < \dist(q,E_2) < 2C_0\ell(Q)$. Note that we can find such a $q \in \cap_{l\in\Z} \cup_{T\in\Delta_l} T$. Then we let $T(Q)$ be the cube in $\Delta_{l(j)}$ containing $q$. We have $T(Q) \in \mathcal{A}_\epsilon$ where
\begin{equation*}
\mathcal{A}_\epsilon = \left\{T\in\Delta : \exists\, q \in T,\, 2C_0\ell(T) < \dist(q,E_2) < 4C_0\epsilon^{-1} \ell(T) \right\}~.
\end{equation*}

For $l \in \Z$ and $T\in\Delta_l$ we set 
\begin{equation*}
\mathcal{D}(T) = \left\{ Q \in \mathcal{B}_\epsilon : T(Q) = T \right\}~.
\end{equation*}
By construction we have $\mathcal{D}(T) \subset \Delta_j$ where $j\in\Z$ is the unique integer such that $l=l(j)$ and $\cup_{Q \in \mathcal{D}(T)} Q \subset B(p_T, 5C_02^j) \cap E_1$. It follows that  $\card \mathcal{D}(T) \lesssim 1$. Furthermore for $R\in\Delta$ and $Q\subseteq R$ such that $Q\in\mathcal{B}_\epsilon$ we have $T(Q) \subset B(p_R, 5C_0\ell(R))$. Hence $\mathcal{B}_\epsilon$ must satisfy a Carleson packing condition whenever $\mathcal{A}_\epsilon$ does.

Let us now check that $\mathcal{A}_\epsilon$ satisfies a Carleson packing condition. Let $p\in E_1$ and set $\mathcal{A}_\epsilon (p) = \left\{T\in \mathcal{A}_\epsilon : p\in T\right\}$. Since the map $\dist(\cdot,E_2)$ is 1-Lipschitz, we have 
\begin{equation*}
C_0 \ell(T) \leq \dist(p,E_2) \leq 5C_0\epsilon^{-1} \ell(T)
\end{equation*}
for all $T\in \mathcal{A}_\epsilon (p)$. Hence $\card \mathcal{A} (p) \lesssim 1$. This easily implies that $\mathcal{A}_\epsilon$ satisfies a Carleson packing condition and concludes the proof of the lemma.
\end{proof}

Let $C \geq 1$, $E\subset\Hek$ be $C$-Ahlfors regular, and $\Delta$ be a system of dyadic cubes on $E$. For each $Q \in \Delta$ recall that $b\beta_E(Q) = b\beta_E(p_Q,2C_0\ell(Q))$, that is,
\begin{displaymath}
b\beta_E(Q) = (2C_0\ell(Q))^{-1} \, \inf_{P \in \mathcal{P}} \left\{\sup_{q \in B(Q) \cap E} \dist(q,P) + \sup_{q \in B(Q) \cap P} \dist(q,E)\right\}~.
\end{displaymath}
Let $\widetilde E \subset \Hek$. Given $p\in\Hek$, $s>0$, set 
\begin{equation*}
I(p,s)= I_{E,\widetilde{E}} (p,s) \quad \text{and} \quad \widetilde{I}(p,s)= I_{\widetilde{E},E} (p,s)~.
\end{equation*}
The next lemma should be compared to~\cite[Part~IV-Lemma~1.20]{MR1251061}.

\begin{lemma} \label{lem:bbetaE-versus-bbetatileE} Let $p\in E\cap \widetilde{E}$ and $Q\in\Delta$ be given with $p\in Q$. Then
\begin{equation} \label{e:bbetaE-versus-bbetatileE}
b\beta_{E}(Q) \leq 3\left( \, b\beta_{\widetilde{E}}(p,6C_0 \ell(Q)) + I(p,6C_0\ell(Q)) + \widetilde{I}(p,6C_0 \ell(Q))\right)~.
\end{equation}
\end{lemma}

\begin{proof}
Let $\epsilon >0$ be given. Let $P\in\calP$ be such that 
\begin{multline*}
\sup_{q \in B(p,6C_0 \ell(Q)) \cap \widetilde{E}} \dist(q,P) + \sup_{q \in B(p,6C_0 \ell(Q)) \cap P} \dist(q,\widetilde{E}) \\
\leq 6C_0 \ell(Q) \, b\beta_{\widetilde{E}}(p,6C_0 \ell(Q)) + \epsilon~.
\end{multline*}
We have
\begin{equation} \label{e:1-bbetaE-versus-bbetatileE}
2C_0\ell(Q) \,  b\beta_{E}(Q) \leq \sup_{q \in B(Q) \cap E} \dist(q,P) + \sup_{q \in B(Q) \cap P} \dist(q,E)
\end{equation}
and $B(Q) \subset B(p, 3C_0 \ell(Q))$.

For $q\in B(Q)$ we denote by $\tilde{q}$ a point in $\widetilde{E}$ such that $d(q,\tilde{q}) \leq \dist (q,\widetilde{E}) + \epsilon$. We have $d(q,\tilde{q}) \leq d(q,p) < 3C_0 \ell(Q)$ and $\tilde{q} \in B(p, 6C_0 \ell(Q))$. If $q \in B(Q) \cap E$, it follows that
\begin{multline*} 
\dist(q,P) \leq \dist(\tilde{q},P) + \dist (q,\widetilde{E}) + \epsilon\\
\leq \sup_{w\in B(p, 6C_0 \ell(Q)) \cap \widetilde{E}} \dist(w,P) + \sup_{\substack{w\in B(p,6C_0\ell(Q))\cap E\\ \dist (w,\widetilde{E}) < 6C_0 \ell(Q)}} \dist (w,\widetilde{E}) + \epsilon~.
\end{multline*}
For $q\in B(Q)\cap P$ we have
\begin{multline*} 
\dist(q,E) \leq d(q,\widetilde{E}) + \dist (\tilde{q},E) + \epsilon\\
\leq \sup_{w\in B(p, 6C_0 \ell(Q)) \cap P} \dist(w,\widetilde{E}) + \sup_{\substack{w\in B(p,6C_0\ell(Q))\cap\widetilde{E}\\ \dist (w,E) <6C_0 \ell(Q)}} \dist (w,E) + \epsilon~.
\end{multline*}
Together with~\eqref{e:1-bbetaE-versus-bbetatileE} we get
\begin{multline*}
2C_0\ell(Q) b\beta_{E}(Q) \\
\leq 6C_0 \ell(Q) \left(b\beta_{\widetilde{E}}(p,6C_0 \ell(Q)) + I(p,6C_0\ell(Q)) + \widetilde{I}(p,6C_0 \ell(Q))\right) + 3\epsilon~.
\end{multline*}
Then~\eqref{e:bbetaE-versus-bbetatileE} follows letting $\epsilon \downarrow 0$.  
\end{proof}

From now on in this section, we assume that $E$ satisfies the assumptions of Theorem~\ref{thm:stability-bwgl}. Given $\epsilon \in(0,1)$ we set 
\begin{equation*}
\Delta_\epsilon=\{Q\in\Delta : b\beta_E(Q)>\epsilon\}~.
\end{equation*}
Since~\eqref{e:f(Q)-vs-f(q,s)} and~\eqref{e:f(q,s)-vs-f(Q)} hold for $f=b\beta_E$, Lemma~\ref{lem:carleson-packing-conditions-continuous-vs-discrete} applies and to prove Theorem~\ref{thm:stability-bwgl} we are going to show that there is $\gamma:(0,1)\rightarrow (0,+\infty)$ depending only on $C$, $\theta$, $C_1$ and $\gamma_1$ such that for each $\epsilon\in (0,1)$ the set $\Delta_\epsilon$ satisfies a $\gamma(\epsilon)$-Carleson packing condition. We will actually show that for each $\epsilon\in(0,1)$ there are $N,\eta >0$ depending only on $\epsilon$, $C$, $\theta$, $C_1$, and $\gamma_1$ such that
\begin{equation} \label{e:sufficient-for-stability-bwgl}
\calHk\left(\left\{ p\in R : \sum_{Q\ni p,\, Q\subseteq R} \chi_{\Delta_\epsilon}(Q) \leq N\right\}\right) \geq \eta \, |R|
\end{equation}
for all $R\in\Delta$. The fact that~\eqref{e:sufficient-for-stability-bwgl} implies a $\gamma(\epsilon)$-Carleson packing condition for $\Delta_\epsilon$ for some $\gamma(\epsilon)>0$ depending only on $\epsilon$, $C$, $\theta$, $C_1$, and $\gamma_1$ follows from~\cite[Part~IV-Lemma~1.12]{MR1251061} which can be verbatim translated in our setting and to which we refer for more details.

To prove~\eqref{e:sufficient-for-stability-bwgl} let $\epsilon\in(0,1)$ and $R\in\Delta$ be given. We know from Definifion~\ref{def:dyadic-cubes}~$(iii)$ that we can find a ball $B$ centered on $R$ with radius $C_0^{-1}\ell(R)$ such that $B\cap E \subset R$. Then we let $\widetilde{E}\in \BWGL(C_1,\gamma_1)$ be such that 
\begin{equation*}
\calH^{2k+1} (E \cap \widetilde{E} \cap B) \geq \theta \, C_0^{-(2k+1)}\, \ell(R)^{2k+1}
\end{equation*}
and hence
\begin{equation} \label{e:big-piece-of-tildeE}
\calHk\left(R\cap \widetilde{E}\right) \geq \theta \, C_0^{-(2k+1)} \, |R|~.
\end{equation}
Set
\begin{equation*}
\begin{aligned}
\widetilde{A}_\epsilon&=\left\{(p,s)\in \widetilde{E}\times \R^+ : b\beta_{\widetilde{E}}(p,6C_0s) > 3^{-1}\epsilon \right\}~,\\
B_\epsilon&=\left\{(p,s)\in E\times \R^+ : I(p,6C_0s) > 3^{-1}\epsilon \right\}~,\\
\widetilde{B}_\epsilon&=\left\{(p,s)\in \widetilde{E} \times \R^+ : \widetilde{I}(p,6C_0s) > 3^{-1}\epsilon \right\}~.
\end{aligned}
\end{equation*}
It follows from Lemma~\ref{lem:bbetaE-versus-bbetatileE}, that, for every $p\in R\cap \widetilde{E}$, 
\begin{equation*}
\sum_{\substack{Q\ni p \\ Q\subseteq R}} \chi_{\Delta_\epsilon}(Q) \leq \sum_{\substack{Q\ni p \\ Q\subseteq R}} \chi_{\widetilde{A}_\epsilon}(p,\ell(Q)) + \chi_{B_\epsilon}(p,\ell(Q)) +  \chi_{\widetilde{B}_\epsilon}(p,\ell(Q))~.
\end{equation*}
Then we get 
\begin{multline*}
\int_{R\cap \widetilde{E}} \sum_{\substack{Q\ni p \\ Q\subseteq R}} \chi_{\Delta_\epsilon}(Q) \,d\calHk(p) \\
\leq \int_{R\cap \widetilde{E}} \int_{0}^{2\ell(R)} \left(\chi_{\widetilde{A}_{\epsilon/2}}(p,s) + \chi_{B_{\epsilon/2}}(p,s) + \chi_{\widetilde{B}_{\epsilon/2}}(p,s)\right)\,  \frac{ds}{s} \,d\calHk(p)~.
\end{multline*}
We have used here the fact that $b\beta_{\widetilde{E}}(p,s') \leq 2 b\beta_{\widetilde{E}}(p,s)$, $I(p,s') \leq 2I(p,s)$ and $\widetilde{I}(p,s') \leq 2 \widetilde{I}(p,s)$ for $s'\leq s \leq 2s'$.

The fact that $\widetilde{E} \in \BWGL(C_1,\gamma_1)$ together with Lemma~\ref{lem:I-carleson-sets} applied to $(E_1,E_2) = (E,\widetilde{E})$ and $(E_1,E_2) = (\widetilde{E},E)$ implies that there is a $\alpha >0$ depending only on $\epsilon$, $C$, $C_1$, and $\gamma_1$ such that
\begin{equation*}
\int_{R\cap \widetilde{E}} \sum_{\substack{Q\ni p \\ Q\subseteq R}} \chi_{\Delta_\epsilon}(Q) \, d\calHk(p) \leq \alpha\, |R|~.
\end{equation*}

Together with~\eqref{e:big-piece-of-tildeE} this implies the existence of $N,\eta >0$ depending only on $\epsilon$, $C$, $\theta$, $C_1$, and $\gamma_1$ such that~\eqref{e:sufficient-for-stability-bwgl} holds and this concludes the proof of Theorem~\ref{thm:stability-bwgl}. 

\begin{remark} The arguments given in this section work in a similar way when considering  the bilateral $\beta$-numbers for vertical hyperplanes $b\beta_{v,E}$ in place of the $b\beta_E$'s.
\end{remark} 

\begin{remark} \label{rk:stability-Euclidean-BWGL} It is clear from the proof of Theorem~\ref{thm:stability-bwgl} that it can be rephrased in the Euclidean setting for Ahlfors regular sets with arbitrary dimension to get the following

\begin{theorem} \label{thm:stability-Euclidean-BWGL} Let $0<d<n$ be integers. Let $C \geq 1$ and $E\subset\R^n$ be $C$-Ahlfors regular with dimension $d$. Assume that there are $\theta>0$, $C_1 \geq 1$, and $\gamma_1:(0,1)\rightarrow (0,+\infty)$ such that for each $p\in E$, $r>0$, there is $\widetilde{E}$ with the following properties. First $ \widetilde{E}$ is $C_1$-Ahlfors regular with dimension $d$. Second for each $\epsilon\in (0,1)$ the set $\{(p,s)\in \widetilde{E} \times (0,+\infty) : b\beta_{\widetilde{E}}(p,s) > \epsilon \}$ is $\gamma_1(\epsilon)$-Carleson. Third 
\begin{equation*}
\calH^{d} (E \cap \widetilde{E} \cap B(p,r)) \geq \theta r^{d}~.
\end{equation*}
Then $E$ satisfies the bilateral weak geometric lemma. More precisely for each $\epsilon\in (0,1)$ the set $\left\{(p,s)\in E \times (0,+\infty) : b\beta_{E}(p,s) > \epsilon \right\}$ is $\gamma(\epsilon)$-Carleson for some $\gamma(\epsilon)$ depending only on $C$, $\theta$, $C_1$ and $\gamma_1$.
\end{theorem}

In the above statement Ahlfors regular sets with dimension $d$ in $\R^n$, Carleson sets, and bilateral $\beta$-numbers are the obvious analogs of the notions considered in this paper where $\Hek$ is replaced by $\R^n$, the Kor\'anyi distance by the Euclidean one, the exponent $2k+1$ in~\eqref{e:def-Carleson-set} by $d$, and affine hyperplanes in~\eqref{e:def-bbeta} by affine $d$-planes, see~\cite[Definitions~1.13,~1.69,~and~(2.1)]{MR1251061}. 

Although the stability under the big pieces functor for several quantitative geometric properties of subsets of $\R^n$ can be found in~\cite[Part~IV]{MR1251061}, the stability of the bilateral weak geometric lemma in the Euclidean setting was not available in the literature so far, at least to our knowledge, and may have its own interest.
\end{remark}


\section{BPiLG implies BWGL} \label{sect:BPILG-implies-BWGL}

The following quantified version of Theorem~\ref{thm:main} follows as an immediate consequence of Theorem~\ref{thm:bwgl-for-intrinsic-Lip-graphs} and Theorem~\ref{thm:stability-bwgl}. We refer to Section~\ref{subsect:BPiLG} for the definition of $\BPiLG(C,\lambda,\theta)$ and to the end of Section~\ref{subsect:def-bwgl} for the definition of $\BWGL(C,\gamma)$.

\begin{theorem} \label{thm:BPILG-implies-BWGL}
For every $C\geq 1$, $\lambda \in (0,1)$, $\theta >0$, there is $\gamma:(0,1)\rightarrow(0,+\infty)$ such that $\BPiLG(C,\lambda,\theta) \subset \BWGL(C,\gamma)$.
\end{theorem}

\begin{remark} \label{rk:final-rk} Other notions of quantitative rectifiability in $\Hek$ have been considered in the recent literature such as the existence of corona decompositions~\cite[Section~9]{MR3815462}. The relationship between BPiLG, BWGL and the existence of corona decompositions will be studied in a forthcoming paper.
\end{remark}


\bibliographystyle{plain}
\bibliography{references}

\end{document}